\documentclass[11pt,a4paper]{article}
\usepackage[T1]{fontenc}
\usepackage{amsfonts,amssymb,amsmath,textcomp,amsthm}
\usepackage{indentfirst}                       
           
\usepackage{ae}                                
\usepackage{geometry}
\usepackage{graphicx}
\usepackage{hyperref}
\usepackage{float}
\usepackage[dvipsnames]{xcolor}
\usepackage{fancyhdr}
\newtheorem{teo}{Theorem}[section]

\newtheorem{lema}[teo]{Lemma}
\newtheorem{cor}[teo]{Corollary}
\newtheorem{prop}[teo]{Proposition}
\theoremstyle{definition}
\newtheorem{defi}[teo]{Definition}
\newtheorem{exem}[teo]{Example}

\newtheorem{obs}[teo]{Remark}
\newtheorem{conj}[teo]{Conjecture}

\def\blfootnote{\xdef\@thefnmark{}\@footnotetext} 

\date{}

\title{\Large \textbf{HALF-ISOMORPHISMS OF AUTOMORPHIC LOOPS}}
\begin{document}

\maketitle

\vspace{-1.6cm}
\begin{center}
$\begin{array}{cc}
\textrm{MARIA DE LOURDES MERLINI GIULIANI} & \textrm{GILIARD SOUZA DOS ANJOS}\\[0.2cm]
\textrm{Centro de Matem\'atica, Computa\c c\~ao e Cogni\c c\~ao} & \textrm{Instituto de Matem\'atica e Estat\'istica}\\
\textrm{Universidade Federal do ABC} & \textrm{Universidade de S\~ao Paulo}\\
\textrm{Av. dos Estados, 5001, 09210-580} & \textrm{Rua do Mat\~ao, 1010, 05508-090}\\
\textrm{Santo Andr\'e - SP, Brazil} & \textrm{S\~ao Paulo - SP, Brazil}\\
\textrm{maria.giuliani@ufabc.edu.br}&\textrm{giliard.anjos@unesp.br}
\end{array}$
\end{center}

\begin{abstract} 
\noindent{}Automorphic loops are loops in which all inner mappings are automorphisms. This variety of loops includes groups and commutative Moufang loops. 
A half-isomorphism $f : G \longrightarrow K$ between multiplicative systems $G$ and $K$ 
is a bijection from $G$ onto $K$ such that $f(ab)\in\{f(a)f(b), f(b)f(a)\}$ for any $a,b\in G$. A half-isomorphism is trivial when it is either an isomorphism or an anti-isomorphism. Consider the class of automorphic loops such that the equation $x\cdot(x\cdot y) = (y\cdot x)\cdot x $ is equivalent to $x\cdot y = y\cdot x$. Here we show that this class of loops includes automorphic loops of odd order and uniquely $2$-divisible. Furthermore, we prove that every half-isomorphism between loops in that class is trivial.

\end{abstract}

\noindent{}{\it Keywords}: half-isomorphism, half-automorphism, automorphic loop.

\section{Introduction}

A loop $( L,.) $ is a set $L$ with a binary operation $ (\cdot) $ such that for each $ a, b \in L $ there exist unique elements $x, y \in L$ such that $ a \cdot x \, = \, b $ and $ y \cdot a \,= \, b$, and there is an identity element $ 1 $ such that 
$ 1 \cdot x = x = x \cdot  1$ for every $ x \in L $. The operation $ \cdot$ is denoted by juxtaposition for the sake of convenience, so $(L, \cdot)$ is denoted by $L$ and $x \cdot y$ by $xy$.

For any $a \in L$, the \textit{right translation} and \textit{left translation} by $a$ are the bijections $R_a: L \rightarrow L $; $ x \mapsto x \cdot a $ and $L_a: L \rightarrow L$; $x \mapsto a \cdot x$ respectively.
These bijections generate the \textit{multiplication group} $ Mlt(L)= \left\langle R_a,L_a ; a \in L \right\rangle $. The \textit{inner mapping group},  denoted by $Inn(L)$,  is the subgroup of  $ Mlt(L)$ containing the elements that stabilize the identity element. 
Automorphic loops (or $A$-loops in  Bruck's notation \cite{BP56}) are loops in which every inner mapping is an automorphism, or $Inn(L)$ is a subgroup of $Aut(L)$, where $Aut(L)$ is the automorphism group of $L$. It is well known that automorphic loops are power associative, i.e., each element generates a cyclic group. \\

Let $(L,*)$ and $(L',\cdot)$ be loops. A bijection $f : L \longrightarrow L'$ is called a \emph{half-isomorphism} if 
\begin{center}
$f(x*y)\in\{f(x)\cdot f(y),f(y)\cdot f(x)\}$, for any $ x, y \in L$.  
\end{center}
A \emph{half-automorphism} and a \emph{half-homomorphism} are defined in a similar manner.
We say that a half-isomorphism is \emph{trivial} when it is either an isomorphism or an anti-isomorphism. Likewise, a half-isomorphism is \textit{nontrivial} when it is neither an isomorphism nor an anti-isomorphism (In \cite{KSV}, the word ``proper'' is used instead).\\

All half-isomorphisms are trivial for groups, as was shown by Scott \cite{Sco57}, but not for loops. For example, in the case of the variety of Moufang loops, there exist loops of even order with nontrivial half-isomorphism \cite{GG13,GPS17} as well as loops of odd order in which all half-isomorphisms are trivial \cite{GG1}. 

These and more recent studies of half-isomorphisms in different classes of loops have shown that the behavior of these functions is rather unpredictable. For example, investigating dihedral automorphic loops, which is a class of automorphic loops of even order, we proved that there exist half-isomorphisms which are neither isomorphisms nor anti-isomorphisms \cite{GA19}. Another case is the case of the automorphic loops of odd order of the class called Lie automorphic loops for which we proved that all half-automorphisms are trivial \cite{GA192}. 

In the current paper, we prove that if $Q$ and $Q'$ are automorphic loops, and in $Q'$ the equation $x\cdot(x\cdot y) = (y\cdot x)\cdot x $ is equivalent to the equation $x\cdot y = y\cdot x$, then every half-isomorphism from $Q$ onto $Q'$ is trivial. The conjecture, presented in \cite{GA192}, that every half-isomorphism between automorphic loops of odd order is trivial, is a corollary of this result.

\section{Preliminaries}
\label{sec2}

Part of the investigations for this work (section 3) were assisted by Prover9 \cite{mace}, an automated computational tool for equational reasoning. However, the raw output is usually over-complicated and here we present only "humanized" versions on these proofs. Also, the LOOPS package \cite{NV1} for GAP \cite{gap} was used.

In this section, the required definitions and basic results are stated. For general facts about loop theory we recommend \cite{B71,P90}. 

Let $L$ be a loop and $x,y\in L$. Consider the following bijections:

\begin{eqnarray}
R_{(x,y)} = R_xR_yR^{-1}_{(xy)},\\
L_{(x,y)} = L_xL_yL^{-1}_{(yx)},\\
T_x = R_xL^{-1}_x.
\end{eqnarray}

It is well-known that the inner mapping group of $L$ is generated by these bijections, that is:

\begin{center}
$Inn(L) = \langle R_{(x,y)},L_{(x,y)}, T_x \,\,|\,\, x,y\in L \rangle$.
\end{center}

The following lemma is straightforward.

\begin{lema} 
\label{lema20}
Let $L$ be a loop, $x,y,z\in L$ and $\varphi \in Aut(L)$. Then

(a) $(y)L^{-1}_x\varphi = (y)\varphi L^{-1}_{(x)\varphi}$,

(b) $(y)R^{-1}_x\varphi = (y)\varphi R^{-1}_{(x)\varphi}$,

(c) $x (y)L^{-1}_x = y$,

(d) $(y)R^{-1}_x x = y$,

(e) $xy=yx$ if and only if $(y)T_x = y$.
\end{lema}

Let $S$ be a non empty subset of $L$ and $K = \{H\leq L\,\,|\,\, S \subset H\}$. The \emph{subloop of $L$ generated by $S$}, denoted by $\langle S \rangle$, is the following subloop of $L$:

\begin{center}
$\langle S \rangle = \bigcap_{H\in K} H$.
\end{center}

\begin{prop}
\label{prop21} (\cite[Thereom $2.4$]{BP56}) Let $Q$ be an automorphic loop. Then

(a) Every maximal commutative subset of $Q$ is a subloop of $Q$.

(b) Every maximal associative subset of $Q$ is a subgroup of $Q$.

(c) $Q$ is power associative, that is, $\langle x \rangle$ is associative, for every $x\in Q$.
\end{prop}

If $Q$ is non-commutative (non-associative), then by Zorn's Lemma we have that every commutative (associative) subset of $Q$ is contained in at least one maximal commutative (associative) subset of $Q$. Then the next result is a direct consequence of Proposition \ref{prop21}.

\begin{cor}
\label{cor21} Let $Q$ be an automorphic loop and $S$ be a non empty subset of $Q$. If $S$ is commutative (associative), then $\langle S \rangle$ is commutative (associative).
\end{cor}

\begin{prop}
\label{prop22} (\cite[Theorem $2.6$]{BP56}) Let $Q$ be an automorphic loop. Then, for all $x,y,z\in Q$ and $m,n,r\in \mathbb{Z}$:

(a) $x^m(x^ny) = x^n(x^my)$,

(b) $(yx^m)x^n = (yx^n)x^m$,

(c) $(x^my)x^n = x^m(yx^n)$,

(d) $x((yx)(zx)) = ((xy)(xz))x$.
\end{prop}

As a direct consequence of the items (a), (b) and (c) of the above Proposition, we have the following result.

\begin{cor}
\label{cor22} Let $Q$ be an automorphic loop and $x\in Q$. Then $\langle L_{x^m},R_{x^n}\,\,|\,\, m,n \in \mathbb{Z} \rangle$ is an abelian group. In particular, we have for all $m\in \mathbb{Z}$:

\begin{equation}
\label{tx1}
T^{m}_x = R^m_xL^{-m}_x = L^{-m}_xR^m_x.
\end{equation}
\end{cor}

\begin{obs} We will omit the parenthesis when occur one of the items (a), (b) and (c) of  Proposition \ref{prop22}. For example, we will write only $x^{-1}yx$ instead of $(x^{-1}y)x$ or $x^{-1}(yx)$.
\end{obs}

\begin{prop}
\label{prop23} (\cite[Theorem $7.5$]{JKNV11}) Every automorphic loop has the antiautomorphic inverse property (AAIP), that is, the following identity holds:

\begin{equation}
\label{aaip}
(xy)^{-1} = y^{-1}x^{-1} \tag{AAIP}
\end{equation}
\end{prop}

The next result is a consequence of the (AAIP).

\begin{cor}
\label{corprop23} If $Q$ is an automorphic loop and $x,y \in Q$, then

\begin{equation}
\label{aaip2}
((y)L^{-1}_x)^{-1} = (y^{-1})R^{-1}_{x^{-1}}.
\end{equation}
\end{cor}

\begin{prop}
\label{prop24} (\cite[Lemma $2.7$]{KKPV16}) In an automorphic loop $Q$, the following holds, for all $x\in Q$:

\begin{equation}
\label{tx2}
T^{-1}_x = T_{x^{-1}}.
\end{equation}
\end{prop}

Define $\gamma: Q \to Q$ by $(x)\gamma = x^2$. We say that $Q$ is \emph{uniquely $2$-divisible} if $\gamma$ is a bijection. In this case, we define $x^{1/2} = (x)\gamma^{-1}$, for every $x\in Q$. Note that, if $\varphi \in Aut(Q)$ and $x\in Q$, then

\begin{equation}
\label{x12}
(x^{1/2})\varphi = ((x)\varphi)^{1/2}
\end{equation}

\begin{prop}
\label{prop25} (\cite[Corollary $4.7$]{KKPV16}) A finite automorphic loop is uniquely $2$-divisible if and only if it has odd order.
\end{prop}

A half-isomorphism $f:L \rightarrow L'$ is called \emph{special} if the inverse mapping $f^{-1}:L' \rightarrow L$ is also a half-isomorphism. Every trivial half-isomorphism is special. The next example presents a half-isomorphism that is not special.

\begin{exem}
\label{ex1} Let $L = \{1,2,...,7\}$ and consider the following Cayley tables:

\begin{center}
\begin{minipage}{.4\textwidth}
 \centering
\begin{tabular}{c|ccccccc}
$*$&1&2&3&4&5&6&7\\
\hline
1&1&2&3&4&5&6&7\\
2&2&3&4&5&6&7&1\\
3&3&4&5&6&7&1&2\\
4&4&5&6&7&1&2&3\\
5&5&6&7&1&2&3&4\\
6&6&7&1&2&3&4&5\\
7&7&1&2&3&4&5&6\\
\end{tabular}

\end{minipage} 
 \begin{minipage}{.4\textwidth}
\begin{tabular}{c|ccccccc}
$\cdot$&1&2&3&4&5&6&7\\
\hline
1&1&2&3&4&5&6&7\\
2&2&3&7&5&6&1&4\\
3&3&4&5&6&7&2&1\\
4&4&5&6&7&1&3&2\\
5&5&6&4&1&2&7&3\\
6&6&7&1&2&3&4&5\\
7&7&1&2&3&4&5&6\\
\end{tabular}

\end{minipage}
\end{center}

One can see that $C_7 = (L,*)$ is the cyclic group of order $7$ and $L' = (L,\cdot)$ is a non-associative loop. Defining $f:C_7 \to L'$; $f(x) = x$, then it is not difficult to see that $f$ is a half-isomorphism. Furthermore, 

\begin{center}
$f(3*2) = 4 = f(3)\cdot f(2)\not = f(2)\cdot f(3)$ \hspace{0.2cm} and  \hspace{0.2cm} $f(3*6) = 1 = f(6)\cdot f(3)\not = f(3)\cdot f(6)$,
\end{center}

then $f$ is a nontrivial half-isomorphism. Now note that $f^{-1}(2\cdot 3) = 7\not \in \{2*3,3*2\}$, and so $f^{-1}$ is not a half-isomorphism.\qed
\end{exem}

\begin{prop}
\label{prop26} (\cite[Theorem $2.5$]{GA192})
Let $L$ and $L'$ be loops and $f:L \rightarrow L'$ be a half-isomorphism. Then the following statements are equivalent:

(a) $f$ is special.

(b) $\{f(x*y),f(y*x)\} = \{f(x)\cdot f(y), f(y)\cdot f(x)\}$ for any $x,y\in L$.

(c) For all $x,y\in L$ such that $x*y = y*x$, we have $f(x)\cdot f(y) = f(y)\cdot f(x)$.
\end{prop}

\begin{prop}
\label{prop27} Let $L$ and $L'$ be power associative loops and $f:L \rightarrow L'$ be a half-isomorphism. Then $f(x^n) = f(x)^n$, for all $x\in L$ and $n\in \mathbb{Z}$.
\end{prop}
\begin{proof} By definition, we have  $f(1_L) = f(1_L*1_L)=f(1_L)\cdot f(1_L)$, and then $f(1_L) = 1_{L'}$. Let $x\in L$ and $n\in \mathbb{N}^*$. Suppose that $f(x^n) = f(x)^n$. Since $L$ and $L'$ are power associative, we have that

\begin{center}
$f(x^{n+1}) = f(x^n*x)\in \{f(x)^n\cdot f(x), f(x)\cdot f(x)^n\} = \{f(x)^{n+1}\}$.
\end{center}

Then $f(x^m) = f(x)^m$, for all $m\in \mathbb{N}$. Now, for $n\in \mathbb{N}$, we have that

\begin{center}
$1_{L'} = f(x^n*x^{-n})\in \{f(x)^n\cdot f(x^{-n}), f(x^{-n})\cdot f(x)^n\}$,
\end{center}

and then $f(x^{-n}) = f(x)^{-n}$ since $L'$ is power associative.
\end{proof}

\section{Some results about automorphic loops}

In this section, we will prove that for every uniquely $2$-divisible automorphic loop the equation\\ $x\cdot(x\cdot y) = (y\cdot x)\cdot x $ is equivalent to the equation $x\cdot y = y\cdot x$. First, we need to prove some identities for automorphic loops.

\begin{lema}
\label{lema31}
In an automorphic loop $Q$, the following identities hold:

\begin{eqnarray}
\label{eql31a}
(xy)^2 =(x\cdot (y)R^{-1}_{x^{-1}})y\\
\label{eql31a2}
(xy)^2 =x \cdot(xy)L^{-1}_{y^{-1}}\\
\label{eql31b}
(xyx^{-1})^2 = (xy)(yx^{-1})\\
\label{eql31c}
(x^2)R^{-1}_y = (x)R^{-1}_{((x)R^{-1}_y)^{-1}}\\
\label{eql31d}
(yx)x^{-1} = x^{-1}(xy)\\
\label{eql31e}
x^2= (x)R^{-1}_y \cdot (x)L^{-1}_{y^{-1}}\\
\label{eql31f}
(y^2x)x^{-1} = (yx)(x^{-1}y)\\
\label{eql31g}
(y^2x)x^{-1} = (yx^{-1})(xy). 
\end{eqnarray}

\end{lema}
\begin{proof} (i) Consider $\varphi = R^{-1}_{y} R_{x^{-1}} L_{xy}$. We have that $(1)\varphi = (xy)(y^{-1}x^{-1}) = 1$, where we used (AAIP) in the second equality, and then $\varphi \in Aut(Q)$. Furthermore, $(xy)\varphi =(xy)(xx^{-1}) = xy$, and so $((xy)^2)\varphi = (xy)^2$. Thus

\begin{center}
$((xy)^2)R^{-1}_{y} R_{x^{-1}} L_{xy} = (xy)L_{xy}$,
\end{center}

and hence $(xy)^2 = (y)L_xR^{-1}_{x^{-1}}R_y$. Since $R^{-1}_{x^{-1}}L_x = L_xR^{-1}_{x^{-1}}$ (Corollary \ref{cor22}), we get \eqref{eql31a}.

(ii) Consider $\varphi = L^{-1}_{x} L_{y^{-1}} L_{xy}$. We have that $\varphi(1) = (xy)
(y^{-1}x^{-1}) = 1$, where we use (AAIP) in the second equality, and then $\varphi \in Aut(Q)$. Furthermore, $\varphi(xy) =(xy)(y^{-1}y) = xy$, and so $((xy)^2)\varphi = (xy)^2$. Thus

\begin{center}
$((xy)^2) L^{-1}_{x} L_{y^{-1}} L_{xy} = (xy)L_{xy}$,
\end{center}

and hence $(xy)^2 = (xy) L^{-1}_{y^{-1}} L_x$ and we have \eqref{eql31a2}.

(iii) Putting $y = vx^{-1}$ in \eqref{eql31a}, we get

\begin{center}
$(xvx^{-1})^2 = (x(vx^{-1})R^{-1}_{x^{-1}})(vx^{-1})$,
\end{center}

and then \eqref{eql31b} follows from the fact that $(vx^{-1})R^{-1}_{x^{-1}} = v$.

(iv) Putting $x = (u)R^{-1}_{y}$ in \eqref{eql31a}, we get

\begin{center}
$((u)R^{-1}_{y}y)^2 =((u)R^{-1}_{y} (y)R^{-1}_{((u)R^{-1}_{y})^{-1}})y$
\end{center}

By Lemma \ref{lema20} (d), we have $u = (u)R^{-1}_{y}y$, and then

\begin{center}
$(u^2)R^{-1}_{y} =  (u)R^{-1}_{y} (y)R^{-1}_{((u)R^{-1}_{y})^{-1}}$,
\end{center}

and using Lemma \ref{lema20} (d) again and Proposition \ref{prop22} we get

\begin{center}
$(u^2)R^{-1}_{y}((u)R^{-1}_{y})^{-1} =(u)R^{-1}_{y} (y)R^{-1}_{((u)R^{-1}_{y})^{-1}}((u)R^{-1}_{y})^{-1} = (u)R^{-1}_{y} y = u$,
\end{center}

and hence we have \eqref{eql31c}.

(v) By \eqref{tx2}, $(y)T_x = (y)T^{-1}_{x^{-1}}$, and then $(y)R_xL^{-1}_x = (y)L_{x^{-1}}R^{-1}_{x^{-1}}$. Thus $(y)R_xL^{-1}_xR_{x^{-1}} = (y)L_{x^{-1}}$ and hence $(y)R_xR_{x^{-1}}L^{-1}_x = (y)L_{x^{-1}}$ by Corollary \ref{cor22}. It follows that $(y)R_xR_{x^{-1}} = (y)L_{x^{-1}}L_x = (y)L_xL_{x^{-1}}$, where we used Corollary \ref{cor22} in the second equality, and therefore we have \eqref{eql31d}.

(vi) Putting $x = (u)R^{-1}_{y}$ in \eqref{eql31a2}, we get

\begin{center}
$((u)R^{-1}_{y}y)^2 =(u)R^{-1}_{y} \cdot ((u)R^{-1}_{y}y)L^{-1}_{y^{-1}}$,
\end{center}

and then \eqref{eql31e} follows from the fact that $(u)R^{-1}_{y}y = u$ (Lemma \ref{lema20} (d)).

(vii) Let $u,v\in Q$ and $\phi = R_uR_{u^{-1}}$. It is clear that $(1)\phi =1$, and then $\phi \in Aut(Q)$. Thus $(v^2)\phi = ((v)\phi)^2$, and we have $((vu)u^{-1})^2 = (v^2u)u^{-1}$. By \eqref{eql31e}, we have

\begin{equation}
\label{eqlema31a}
(v^2u)u^{-1}= ((vu)u^{-1})R^{-1}_z \cdot ((vu)u^{-1})L^{-1}_{z^{-1}}, \textrm{ for all } z\in Q.
\end{equation}

(vii.1) Putting $z = u^{-1}$ in \eqref{eqlema31a}, we get 

\begin{center}
$(v^2u)u^{-1}= ((vu)u^{-1})R^{-1}_{u^{-1}} \cdot ((vu)u^{-1})L^{-1}_{u} = (vu) \cdot ((vu)u^{-1})L^{-1}_{u} $.
\end{center}

Using \eqref{eql31d} and Proposition \ref{prop22} (a) we get $((vu)u^{-1})L^{-1}_{u} = (u^{-1}(uv))L^{-1}_{u} = (u(u^{-1}v))L^{-1}_{u} = u^{-1}v$. Therefore $(v^2u)u^{-1}= (vu)(u^{-1}v)$ and we have \eqref{eql31f}.

(vii.2) Putting $z = u$ in \eqref{eqlema31a}, we get 

\begin{center}
$(v^2u)u^{-1}= ((vu)u^{-1})R^{-1}_u \cdot ((vu)u^{-1})L^{-1}_{u^{-1}} $.
\end{center}

Using Proposition \ref{prop22} (b) we get $((vu)u^{-1})R^{-1}_u = ((vu^{-1})u)R^{-1}_u = vu^{-1}$. Using \eqref{eql31d} we get $((vu)u^{-1})L^{-1}_{u} = (u^{-1}(uv))L^{-1}_{u} = uv$. Therefore $(v^2u)u^{-1}= (vu^{-1})(uv)$ and we have \eqref{eql31g}.
\end{proof}

\begin{cor}
\label{lema32}
Let $Q$ be an automorphic loop and $x,y\in Q$. Consider that one of the following hold:

(a) $x^{-1}(xy^2) = (xy)(x^{-1}y)$ or  \hspace{1cm}(b) $x^{-1}(xy^2) = (x^{-1}y)(xy)$. 

Then $xy=yx$.
\end{cor}
\begin{proof}
(a) Using \eqref{eql31d} we get $(y^2x)x^{-1} =  (xy)(x^{-1}y)$, and then $(xy)(x^{-1}y) = (yx)(x^{-1}y)$ by \eqref{eql31f}. Hence $xy=yx$.

(b) Using \eqref{eql31d} we get $(y^2x)x^{-1} =  (x^{-1}y)(xy)$, and then $(x^{-1}y)(xy) = (yx^{-1})(xy)$ by \eqref{eql31g}. Thus $x^{-1}y = yx^{-1}$ and from Corollary \ref{cor21} we get $xy=yx$.
\end{proof}

\begin{lema}
\label{lema34}
In an automorphic loop $Q$, the following identities hold:

\begin{eqnarray}
\label{e177}
((x)R^{-1}_{((x)R^{-1}_y)^{-1}})^{-1}\cdot (x)R^{-1}_y = ((x)R^{-1}_y)^{-1}\cdot y^{-1}\\
\label{e245}
(((x)R^{-1}_{xy})^{-1}R^{-1}_{x^{-1}})^{-1}\cdot x= (x)R^{-1}_{yx}\\
\label{e215}
((x\cdot (y)R^{-1}_z)z)R^{-1}_y =(xz)R^{-1}_y \cdot ((z)R^{-1}_y)^{-1}\\
\label{e216}
(xy)R^{-1}_z \cdot ((y)R^{-1}_z)^{-1} = ((x)R^{-1}_z \cdot ((y)R^{-1}_z)^{-1})R_y\\
\label{e258}
x^{-1}((x)R^{-1}_{(y)T^{-1}_x} \cdot ((y)R^{-1}_xR^{-1}_{(y)T^{-1}_x})^{-1}) = (x)R^{-1}_{(y)T^{-1}_x} \cdot ((y)R^{-1}_{(y)T^{-1}_x})^{-1}\\
\label{e254}
((x)R^{-1}_{xy})^{-1}R^{-1}_{(y)R^{-1}_{xy} \cdot ((x)R^{-1}_{xy})^{-1}} = x.
\end{eqnarray}
\end{lema}
\begin{proof}
(i) We have that 

\begin{flushleft}
$\begin{array}{lcl}
((x)R^{-1}_y)^{-1}\cdot (x)R^{-1}_{((x)R^{-1}_y)^{-1}} \cdot ((x)R^{-1}_y)^{-1} & = & ((x)R^{-1}_y)^{-1}\cdot x \\
& =&  ((x)R^{-1}_y)^{-1}((x)R^{-1}_y\cdot y)\\
& \stackrel{\eqref{eql31d}}{=}& ( y\cdot (x)R^{-1}_y) ((x)R^{-1}_y)^{-1}.
\end{array}$
\end{flushleft}

Then $((x)R^{-1}_y)^{-1}\cdot (x)R^{-1}_{((x)R^{-1}_y)^{-1}} = y\cdot (x)R^{-1}_y$, and applying (AAIP) we obtain \eqref{e177}.\\

(ii) By Lemma \ref{lema20} (b), we have that $(x)R^{-1}_{xy}T_x = (x)T_xR^{-1}_{(xy)T_x}$. Since $(x)T_x = x$ and $(xy)T_x = yx$, we get $(x)R^{-1}_{xy}T_x = (x)R^{-1}_{yx}$. By \eqref{tx1}, $T_x = L^{-1}_xR_x$, and then $(x)R^{-1}_{xy}L^{-1}_xR_x = (x)R^{-1}_{yx}$. Using \eqref{aaip2} in the left side of this equation, we get

\begin{center}
$(((x)R^{-1}_{xy})^{-1}R^{-1}_{x^{-1}})^{-1}R_x = (x)R^{-1}_{yx}$.
\end{center}

Hence we have \eqref{e245}.\\

(iii) Note that $R_{((y)R^{-1}_z,z)} = R_{(y)R^{-1}_z}R_zR^{-1}_y$ since $y = (y)R^{-1}_z\cdot z$. Applying this automorphism to $x = (x)R^{-1}_{(y)R^{-1}_z}\cdot (y)R^{-1}_z$, we get

\begin{flushleft}
$\begin{array}{lcl}
((x\cdot (y)R^{-1}_z)z)R^{-1}_y & = &(x)R_{(y)R^{-1}_z}R_zR^{-1}_y\\
& =&((x)R^{-1}_{(y)R^{-1}_z})R_{(y)R^{-1}_z}R_zR^{-1}_y \cdot ((y)R^{-1}_z)R_{(y)R^{-1}_z}R_zR^{-1}_y  \\
& =& (x)R_zR^{-1}_y \cdot ((y)R^{-1}_z)R_{(y)R^{-1}_z}R_zR^{-1}_y\\
& =& (x)R_zR^{-1}_y \cdot (((y)R^{-1}_z)^{-1}R_{(y)R^{-1}_z}R_zR^{-1}_y)^{-1}\\
& =& (x)R_zR^{-1}_y \cdot ((1)R_zR^{-1}_y)^{-1}.
\end{array}$
\end{flushleft}

Then we have \eqref{e215}.\\

(iv) Applying the automorphism $R_{((z)R^{-1}_y,y)} = R_{(z)R^{-1}_y}R_yR^{-1}_z$ to $x = (x)R^{-1}_y \cdot y$, we get

\begin{equation}
\label{eqlema34b}
(x)R_{(z)R^{-1}_y}R_yR^{-1}_z = ((x)R^{-1}_y)R_{(z)R^{-1}_y}R_yR^{-1}_z \cdot (y)R_{(z)R^{-1}_y}R_yR^{-1}_z
\end{equation}

By \eqref{e215}, we have

\begin{center}
$(x)R_{(z)R^{-1}_y}R_yR^{-1}_z = (xy)R^{-1}_z \cdot ((y)R^{-1}_z)^{-1}$ and\\
 $((x)R^{-1}_y)R_{(z)R^{-1}_y}R_yR^{-1}_z = ((x)R^{-1}_y\cdot y)R^{-1}_z \cdot ((y)R^{-1}_z)^{-1} = (x)R^{-1}_z \cdot ((y)R^{-1}_z)^{-1}$.
\end{center}

Furthermore, 

\begin{center}
$(y)R_{(z)R^{-1}_y}R_yR^{-1}_z = (y\cdot (z)R^{-1}_y \cdot y)R^{-1}_z = (yz)R^{-1}_z =y$.
\end{center}

Hence \eqref{e216} follows from \eqref{eqlema34b}.\\

(v) From \eqref{e216} we have

\begin{center}
$((u)R^{-1}_v \cdot ((x)R^{-1}_v)^{-1})x =(ux)R^{-1}_v \cdot ((x)R^{-1}_v)^{-1} ,$
\end{center}

and applying (AAIP) we get

\begin{center}
$x^{-1}((x)R^{-1}_v \cdot ((u)R^{-1}_v)^{-1})=(x)R^{-1}_v \cdot ((ux)R^{-1}_v)^{-1}.$
\end{center}

Putting $u = (y)R^{-1}_x$ and $v = (y)T^{-1}_x$ in this equation, we obtain \eqref{e258}.\\

(vi) From \eqref{e215} we have $((u\cdot (y)R^{-1}_x)x)R^{-1}_y =(ux)R^{-1}_y \cdot ((x)R^{-1}_y)^{-1}$. Putting $u = ((y)R^{-1}_x)^{-1}$ in this equation, we get

\begin{center}
$(x)R^{-1}_y =(((y)R^{-1}_x)^{-1}\cdot x)R^{-1}_y \cdot ((x)R^{-1}_y)^{-1}$
\end{center}

\begin{flushleft}
$\begin{array}{cl}
\stackrel{\eqref{e216}}{\implies} &(x)R^{-1}_y = (((y)R^{-1}_x)^{-1}R^{-1}_y \cdot ((x)R^{-1}_y)^{-1})R_x\\
\implies & (x)R^{-1}_yR^{-1}_x = ((y)R^{-1}_x)^{-1}R^{-1}_y \cdot ((x)R^{-1}_y)^{-1}\\
\implies & ((x)R^{-1}_yR^{-1}_x)L^{-1}_{((y)R^{-1}_x)^{-1}R^{-1}_y} =  ((x)R^{-1}_y)^{-1}\\
\implies & (((x)R^{-1}_yR^{-1}_x)L^{-1}_{((y)R^{-1}_x)^{-1}R^{-1}_y})^{-1} =  (x)R^{-1}_y\\
\stackrel{\eqref{aaip2}}{\implies} & ((x)R^{-1}_yR^{-1}_x)^{-1}R^{-1}_{(((y)R^{-1}_x)^{-1}R^{-1}_y)^{-1}} = (x)R^{-1}_y.
\end{array}$
\end{flushleft}

Putting $x = wy$ in this equation and noticing that $(wy)R^{-1}_y = w$, we have

\begin{equation}
\label{e251}
w = ((w)R^{-1}_{wy})^{-1}R^{-1}_{(((y)R^{-1}_{wy})^{-1}R^{-1}_y)^{-1}}.
\end{equation}

From \eqref{e216} we have $(wy)R^{-1}_{wy} \cdot ((y)R^{-1}_{wy})^{-1} = ((w)R^{-1}_{wy} \cdot ((y)R^{-1}_{wy})^{-1})R_y$. Since $(wy)R^{-1}_{wy} = 1$, it follows that

\begin{center}
$((y)R^{-1}_{wy})^{-1}R^{-1}_y = (w)R^{-1}_{wy} \cdot ((y)R^{-1}_{wy})^{-1}$,
\end{center}

and using (AAIP) we get

\begin{center}
$(((y)R^{-1}_{wy})^{-1}R^{-1}_y)^{-1} = (y)R^{-1}_{wy} \cdot ((w)R^{-1}_{wy})^{-1}$.
\end{center}

Putting the above equation in \eqref{e251}, we obtain \eqref{e254}.
\end{proof}

\begin{prop}
\label{prop30}
In an automorphic loop $Q$, the following identity holds:

\begin{equation}
\label{e266}
(x)R^{-1}_{(x)T^{-1}_y}\cdot y = (y)T^{-1}_x
\end{equation}

\end{prop}
\begin{proof} From \eqref{e177} we have $((u)R^{-1}_{((u)R^{-1}_v)^{-1}})^{-1}\cdot (u)R^{-1}_v = ((u)R^{-1}_v)^{-1}\cdot v^{-1}$. If $u = (x)R^{-1}_{xy}$ and $v = (y)R^{-1}_{xy} \cdot ((x)R^{-1}_{xy})^{-1}$, then $(u)R^{-1}_v = x$ by \eqref{e254}, and we get

\begin{center}
$ (((x)R^{-1}_{xy})^{-1}R^{-1}_{x^{-1}})^{-1} \cdot x = x^{-1} \cdot ((y)R^{-1}_{xy} \cdot ((x)R^{-1}_{xy})^{-1})^{-1}$
\end{center}

Using \eqref{e245} in the left side and (AAIP) in the right side of this equation, we obtain

\begin{center}
$(x)R^{-1}_{yx}=x^{-1}((x)R^{-1}_{xy} \cdot ((y)R^{-1}_{xy})^{-1})$.
\end{center}

Putting $w = yx$ in the above equation and noticing that $y = (w)R^{-1}_x$,  we get

\begin{center}
$(x)R^{-1}_{w}=x^{-1}((x)R^{-1}_{x\cdot (w)R^{-1}_x} \cdot ((w)R^{-1}_xR^{-1}_{x\cdot(w)R^{-1}_x})^{-1})$.
\end{center}

Since $(w)T^{-1}_x = x\cdot (w)R^{-1}_x$, from \eqref{e258} we have

\begin{center}
$(x)R^{-1}_{w} = (x)R^{-1}_{(w)T^{-1}_x}\cdot ((w)R^{-1}_{(w)T^{-1}_x})^{-1}$,
\end{center}

and then

\begin{equation}
\label{e259}
x = ((x)R^{-1}_{(w)T^{-1}_x}\cdot ((w)R^{-1}_{(w)T^{-1}_x})^{-1})R_w.
\end{equation}

Putting $z = (w)T^{-1}_x$ in \eqref{e216}, we have 

\begin{center}
$(xw)R^{-1}_{(w)T^{-1}_x} \cdot ((w)R^{-1}_{(w)T^{-1}_x})^{-1} = ((x)R^{-1}_{(w)T^{-1}_x} \cdot ((w)R^{-1}_{(w)T^{-1}_x})^{-1})R_w$,
\end{center}

and putting this in \eqref{e259}, we obtain

\begin{center}
$x = (xw)R^{-1}_{(w)T^{-1}_x} \cdot ((w)R^{-1}_{(w)T^{-1}_x})^{-1}$
\end{center}

\begin{flushleft}
$\begin{array}{cl}
\stackrel{\textrm{Lemma }\ref{lema20} (b)}{\implies}& x = ((xw)T_xR^{-1}_{w} \cdot ((w)T_xR^{-1}_{w})^{-1})T^{-1}_x  \\

\stackrel{(x)T_x=x}{\implies}&x = (xw)T_xR^{-1}_{w} \cdot ((w)T_xR^{-1}_{w})^{-1} \\

\stackrel{(xw)T_x=wx}{\implies}&x = (wx)R^{-1}_{w} \cdot ((w)T_xR^{-1}_{w})^{-1} \\

\implies&x = (x)T^{-1}_{w} \cdot ((w)T_xR^{-1}_{w})^{-1} \\

\implies&(x)L^{-1}_{(x)T^{-1}_{w}} = ((w)T_xR^{-1}_{w})^{-1} \\

\stackrel{\eqref{aaip2}}{\implies}& (x^{-1})R^{-1}_{(x^{-1})T^{-1}_{w}} = (w)T_xR^{-1}_{w}\\

\implies& (x^{-1})R^{-1}_{(x^{-1})T^{-1}_{w}}R_w = (w)T_x\\

\stackrel{\eqref{tx2}}{\implies}& (x^{-1})R^{-1}_{(x^{-1})T^{-1}_{w}}R_w = (w)T^{-1}_{x^{-1}}
\end{array}$
\end{flushleft}
\end{proof}

Let $Q$ be an automorphic loop. For $x,y\in Q$, consider the following condition:

\begin{equation}
\label{co1}
x(xy) = (yx)x \textrm{ if and only if } xy = yx.
\end{equation}

By \eqref{tx1}, we have $T^2_x = R^2_xL^{-2}_x$, and then the above condition  can be rewritten as:

\begin{equation}
\label{co2}
(y)T^2_x = y \textrm{ if and only if } (y)T_x = y.
\end{equation}

\begin{obs}
\label{obs31}
In flexible loops (and therefore in automorphic loops), $xy = yx$ implies that $x(xy) = (yx)x$.
\end{obs}

As a consequence of \eqref{co2} and Corollary \ref{cor21}, we have the following result.

\begin{lema}
\label{lema33} Let $Q$ be an automorphic loop satisfying \eqref{co1} and let $x,y\in Q$.  If $(y)T^2_x = y$, then $\langle x,y\rangle$ is commutative.
\end{lema}

\begin{cor}
\label{cor31} Let $Q$ be an automorphic loop satisfying \eqref{co1} and let $x,y\in Q$.  If $xyx^{-1} = x^{-1}yx$, then $\langle x,y\rangle$ is commutative.
\end{cor}
\begin{proof} By assumption, $(y) R_{x^{-1}}L_x= (y)R_{x} L_{x^{-1}}$, and then $(y)R_{x^{-1}} L_{x^{-1}}^{-1}= (y)R_{x} L_x^{-1}$ by Corollary \ref{cor22}. Thus $(y)T_{x^{-1}} = (y)T_x$, and hence $(y)T^2_x = y$ by  \eqref{tx2}. Therefore the claim follows from Lemma \ref{lema33}.
\end{proof}

\begin{teo}
\label{teo31}
Let $Q$ be an automorphic loop and $x,y\in Q$. Then

\begin{center}
$x(xy) = (yx)x$ if and only if  $x^2y = yx^2.$
\end{center}

\end{teo}
\begin{proof} We have that

\begin{flushleft}
$\begin{array}{lclcl}
x^2y = yx^2 &\iff& (x^2)T^{-1}_y = x^2& \iff&((x)T^{-1}_y)^2 = x^2 \\

&\iff& (x)T^{-1}_y = (x^2)R^{-1}_{(x)T^{-1}_y}& \stackrel{\eqref{eql31c}}{\iff}& (x)T^{-1}_y = (x)R^{-1}_{((x)R^{-1}_{(x)T^{-1}_y})^{-1}} \\

&\iff& x = ((x)T^{-1}_y)((x)R^{-1}_{(x)T^{-1}_y})^{-1}& \stackrel{\textrm{(AAIP)}}{\iff}& x^{-1} = ((x)R^{-1}_{(x)T^{-1}_y})((x)T^{-1}_y)^{-1} \\

&\iff& x^{-1} = ((x)R^{-1}_{(x)T^{-1}_y})((x^{-1})T^{-1}_y) &\iff&(x^{-1})R^{-1}_{(x^{-1})T^{-1}_y} = (x)R^{-1}_{(x)T^{-1}_y}\\

&\iff&(x^{-1})R^{-1}_{(x^{-1})T^{-1}_y} y = (x)R^{-1}_{(x)T^{-1}_y} y& \stackrel{\eqref{e266}}{\iff}& (y)T^{-1}_{x^{-1}} = (y)T^{-1}_x\\

&\iff& (y)T^{-1}_{x^{-1}}T_x = y & \stackrel{\eqref{tx2}}{\iff}& (y)T^{2}_x = y

\end{array}$
\end{flushleft}
\end{proof}

\begin{cor}
\label{cor32}
If $Q$ is an uniquely $2$-divisible automorphic loop, then it satisfies \eqref{co1}.
\end{cor}
\begin{proof} Let $x,y\in Q$ be such that $x(xy) = (yx)x$.  By Theorem \ref{teo31}, we have that $x^2y = yx^2$, and then $(x^2)T_y = x^2$. Since $T_y\in Aut(Q)$, it follows that $(x)T_y = x$ by \eqref{x12}.
\end{proof}

The next result is immediate from Proposition \ref{prop25} and Corollary \ref{cor32}.

\begin{cor}
\label{cor33}
Every automorphic loop of odd order satisfies \eqref{co1}.
\end{cor}

It is clear that every commutative automorphic loop satisfies \eqref{co1} and the direct product of two automorphic loops that satisfy \eqref{co1} also satisfies this condition. Then there are also non-commutative automorphic loops of even whose satisfy \eqref{co1}.

\section{Main results}

In this section, $(Q,*)$ and $(Q',\cdot)$ are automorphic loops, and $f:(Q,*)\to (Q',\cdot)$ is a half-isomorphism. Furthermore, we consider that $Q'$  satisfies \eqref{co1}, that is, for every $x,y\in Q'$ we have

\begin{center}
$x\cdot(x\cdot y) = (y\cdot x)\cdot x \textrm{ if and only if } x\cdot y = y\cdot x.$
\end{center}

\begin{lema}
\label{lema41} Let $x,y\in Q$ be such that $f(x*y) = f(x)\cdot f(y)$ and $f(y*x^{-1}) = f(x^{-1})\cdot f(y)$. Then $f(x)\cdot f(y) = f(y)\cdot f(x)$.
\end{lema}
\begin{proof} By Proposition \ref{prop27}, $f(x^{-1}) = f(x)^{-1}$. Then using the definition of half-isomorphisms:

\begin{center}
$f((x*y)*x^{-1})\in U = \{f(x)\cdot f(y)\cdot f(x)^{-1}, f(x)^{-1}\cdot (f(x)\cdot f(y))\}$,

$f(x*(y*x^{-1}))\in V = \{f(x)\cdot (f(x)^{-1}\cdot f(y)),f(x)^{-1}\cdot f(y) \cdot f(x)\}$.
\end{center}

Since $(x*y)*x^{-1} = x*(y*x^{-1})$, we have that $U\cap V \not = \emptyset$. If $f(x)\cdot f(y)\cdot f(x)^{-1} = f(x)\cdot (f(x)^{-1}\cdot f(y))$, then $f(y)\cdot  f(x)^{-1} =f(x)^{-1}\cdot f(y) $, and so $f(x)\cdot f(y) = f(y)\cdot f(x)$ by Corollary \ref{cor21}. When $f(x)\cdot f(y)\cdot f(x)^{-1} = f(x)^{-1}\cdot f(y) \cdot f(x)$ we have that $f(x)\cdot f(y) = f(y)\cdot f(x)$ by Corollary \ref{cor31}. Now consider that $f(x*y*x^{-1}) = f(x)^{-1}\cdot (f(x)\cdot f(y))$.

By Lemma \ref{lema31} (b), we have that $(x*y*x^{-1})^2 = (x*y)*(y*x^{-1})$, and then

\begin{equation}
\label{eqlema41a}
f((x*y*x^{-1})^2)\in \{(f(x)\cdot f(y))\cdot(f(x)^{-1}\cdot f(y)),(f(x)^{-1}\cdot f(y))\cdot  (f(x)\cdot f(y))\}
\end{equation}

Note that $\varphi = L_{f(x)}L_{f(x)^{-1}}\in Aut(Q')$. Then

\begin{center}
$f((x*y*x^{-1})^2) = (f(x*y*x^{-1}))^2 = ((f(y))\varphi)^2 = (f(y)^2)\varphi = f(x)^{-1}\cdot (f(x)\cdot f(y)^2)$,
\end{center}

and putting this in \eqref{eqlema41a}, we get

\begin{center}
$f(x)^{-1}\cdot (f(x)\cdot f(y)^2) = (f(x)\cdot f(y))\cdot(f(x)^{-1}\cdot f(y))$ or \\$f(x)^{-1}\cdot (f(x)\cdot f(y)^2) = (f(x)^{-1}\cdot f(y))\cdot  (f(x)\cdot f(y))$.
\end{center}

By Corollary \ref{lema32}, we get that $f(x)\cdot f(y) = f(y)\cdot f(x)$ in both cases.
\end{proof}

\begin{prop}
\label{prop41}
$f$ is a special half-isomorphism.
\end{prop}
\begin{proof} Let $x,y\in Q$ be such that $x*y = y*x$. It is suffices to prove that $f(x)\cdot f(y) = f(y)\cdot f(x)$ by Proposition \ref{prop26}. 

The case $f(x*y) = f(y*x) = f(x)\cdot f(y)$ is symmetric to the case $f(y*x) = f(y)\cdot f(x)$, so we only prove the first one. If $f(y*x^{-1}) = f(x^{-1})\cdot f(y)$, then $f(x)\cdot f(y) = f(y)\cdot f(x)$ by Lemma \ref{lema41}. Now consider that $f(y*x^{-1}) = f(y)\cdot f(x^{-1})$. By Corollary \ref{cor21}, we have that $\langle x,y\rangle$ is commutative, and then $y*x^{-1} = x^{-1}*y$ and $f(x^{-1}*y) = f(y)\cdot f(x^{-1})$. Since $f(x*y) = f(x)\cdot f(y)$, using (AAIP) and Proposition \ref{prop27} we get

\begin{center}
$f(y^{-1}*x^{-1}) = (f(x*y))^{-1} = f(y)^{-1}\cdot f(x)^{-1} = f(y^{-1})\cdot f(x^{-1})$.
\end{center}

By Lemma \ref{lema41}, we have that $f(y^{-1})\cdot f(x^{-1}) = f(x^{-1})\cdot f(y^{-1})$, and then using (AAIP) and Proposition \ref{prop27} we get that $f(x)\cdot f(y) = f(y)\cdot f(x)$.
\end{proof}

Since $f$ is special, by Proposition \ref{prop26} we have, for all $x,y\in Q$:

\begin{eqnarray}
\label{sp1}
\textrm{If } x*y = y*x, \textrm{ then } f(x)\cdot f(y) = f(y)\cdot f(x),\\
\label{sp2}
\textrm{If } f(x*y) = f(x)\cdot f(y),  \textrm{ then } f(y*x) = f(y)\cdot f(x),\\
\label{sp3}
\textrm{If } f(x*y) = f(y)\cdot f(x),  \textrm{ then } f(y*x) = f(x)\cdot f(y).
\end{eqnarray}

Let $L,L'$ be flexible loops. A mapping $f$ from $L$ into $L'$ is called a \emph{semi-homomorphism} if $f(xyx) = f(x)f(y)f(x)$, for all $x,y\in L$.

\begin{prop}
\label{prop42}
$f$ is a semi-homomorphism.
\end{prop}
\begin{proof} Let $x,y\in Q$. We have two cases:

(i) $f(x*y) = f(x)\cdot f(y)$. Since $f$ is special, we have $f(y*x) = f(y)\cdot f(x)$ by \eqref{sp2}. Then:

\begin{center}
$f((x*y)*x) \in \{f(x)\cdot f(y)\cdot f(x),f(x)\cdot (f(x)\cdot f(y))\}$,

$f(x*(y*x)) \in \{f(x)\cdot f(y)\cdot f(x),(f(y)\cdot f(x))\cdot f(x)\}$.
\end{center}

Since $Q$ is flexible, either $f(x*y*x)= f(x)\cdot f(y)\cdot f(x)$ or 

\begin{equation}
\label{eqprop42}
f(x*y*x) = f(x)\cdot (f(x)\cdot f(y)) = (f(y)\cdot f(x))\cdot f(x).
\end{equation}

In the second case, we get  $f(x)\cdot f(y) = f(y)\cdot f(x)$ since $Q'$ satisfies \eqref{co1}. Putting this in \eqref{eqprop42}, we get $f(x*y*x)= f(x)\cdot f(y)\cdot f(x)$.

(ii) $f(x*y) = f(y)\cdot f(x)$. This case is analogous to (i).
\end{proof}

\begin{prop}
\label{prop42a}
$Q$ satisfies \eqref{co1}.
\end{prop}
\begin{proof} Let $x,y\in Q$ be such that $x*(x*y) = (y*x)*x$. We have two cases:

(i) $f(x*y) = f(x)\cdot f(y)$. By Proposition \ref{prop42}, we have that 

\begin{center}
$f((x*y)*x)= (f(x)\cdot f(y))\cdot f(x) =  f(x)\cdot (f(y)\cdot f(x))= f(x*(y*x))$. 
\end{center}

Since $f$ is special, we get that $f(x*(x*y)) = f(x) \cdot (f(x)\cdot f(y))$ and $f((y*x)*x) = (f(y)\cdot f(x))\cdot f(x)$ by \eqref{sp2}. Since $x*(x*y) = (y*x)*x$, we get $f(x) \cdot (f(x)\cdot f(y)) =  (f(y)\cdot f(x))\cdot f(x)$, and then $f(x)\cdot f(y) = f(y)\cdot f(x)$ since $Q'$ satisfies \eqref{co1}. Therefore $x*y = y*x$.

(ii) $f(x*y) = f(y)\cdot f(x)$. This case is analogous to (i).
\end{proof}

\begin{defi}(\cite{GGRS16}) Let $x,y,z\in Q$. The triple $(x,y,z)$ is called a \emph{GG-triple} of $f$ if 

(i) $f(x*y) = f(x)\cdot f(y) \not = f(y)\cdot f(x)$,

(ii) $f(x*z) = f(z)\cdot f(x) \not = f(x)\cdot f(z)$.

\end{defi}

\begin{obs}
In Example \ref{ex1}, $(3,2,6)$ is a GG-triple.
\end{obs}

The next lemma is a well-known result about loops. A proof of it can be found in \cite[Lemma~1]{GA19}.

\begin{lema}
\label{lema42} Let $L$ be a loop and $H,K$ be subloops of $L$. If $L = H\cup K$, then either $H = L$ or $K = L$.
\end{lema}

Let $S$ be a non empty subset of a loop $L$. The \emph{commutant} of $S$ is the set 

\begin{center}
$C_L(S) = \{x\in L\,|\, xy=yx, \forall \, y\in S\}$.
\end{center}

If $L$ is an automorphic loop, then $C_L(S)$ is a subloop of $L$ \cite[Proposition $2.10$]{KKPV16}. 

\begin{prop} 
\label{prop43} Let $Q_1$ and $Q_2$ be automorphic loops and $\varphi: Q_1\to Q_2$ be a special half-isomorphism. If $\varphi$ has no GG-triple, then $\varphi$ is trivial.
\end{prop}
\begin{proof} Let $A = \{x\in Q_1\,|\, \varphi(x*u) = \varphi(x)\cdot \varphi(u), \forall\, u\in Q_1\}$ and $B = \{y\in Q_1\,|\, \varphi(y*u) = \varphi(u)\cdot \varphi(y), \forall\, u\in Q_1\}$. Since $\varphi$ has no GG-triple, then $Q_1 = A\cup B$.

For $x\in A$ and $y\in B$, we have by definition that $\varphi(x*y) = \varphi(x)\cdot \varphi(y) = \varphi(y*x)$, and then $x*y=y*x$. Thus $A\subset C_{Q_1}(y)$ and $B\subset C_{Q_1}(x)$, for all $x\in A$ and $y\in B$. Since $Q_1 = A\cup B$, it follows by Lemma \ref{lema42} that, for every $(x,y)\in (A,B)$, either $Q_1 = C_{Q_1}(x)$ or $Q_1 = C_{Q_1}(y)$.

If there exists $x\in A$ such that $C_{Q_1}(x)\not = Q_1$, then $C_{Q_1}(y) = Q_1$, for all $y\in B$, and so $ \varphi(y*u) = \varphi(u)\cdot \varphi(y) = \varphi(y)\cdot \varphi(u)$, for all $u\in Q_1$ and $y\in B$, since $\varphi$ is  special. Thus $B\subset A$, and we get that $Q_1 = A$. Hence $\varphi$ is an isomorphism. Similarly, when $C_{Q_1}(x) = Q_1$, for all $x\in A$, we get that $A\subset B$, and then  $Q_1 = B$ and $\varphi$ is an anti-isomorphism.
\end{proof}

For $x\in Q$, define the mapping $\phi_x = R_xL_{x^{-1}}$. Then $(u)\phi_x = x^{-1}*u*x = u^x$, for all $u\in Q$. Since $(1)\phi_x = 1$, we have that $\phi_x\in Aut(Q)$, and then, for all $u,v\in Q$:

\begin{equation}
\label{phi1}
(u*v)^x = u^x*v^x
\end{equation}

\begin{lema}
\label{lema43} Let $x,y\in Q$.

(a) If $x*y = y*x$, then $f(y^x) = f(y)^{f(x)} = f(y)^{f(x)^{-1}}$.

(b) If $x*y \not = y*x$ and $f(x*y) = f(x)\cdot f(y)$, then $f(y^x) = f(y)^{f(x)}$ and $f(y^{x^{-1}}) = f(y)^{f(x)^{-1}}$.

(c) If $x*y \not = y*x$ and $f(x*y) = f(y)\cdot f(x)$, then $f(y^x) = f(y)^{f(x)^{-1}}$ and $f(y^{x^{-1}}) = f(y)^{f(x)}$.
\end{lema}
\begin{proof} Using the definition of half-isomorphisms, we get

\begin{eqnarray}
\label{eqlema43a}
f(x^{-1}*(x*y)) \in U = \{f(x)^{-1}\cdot f(x*y), f(x*y)\cdot f(x)^{-1}\},\\
\label{eqlema43b}
f(x*(x^{-1}*y)) \in V = \{f(x)\cdot f(x^{-1}*y), f(x^{-1}*y)\cdot f(x)\}.
\end{eqnarray}

Note that $x^{-1}*(x*y) = x*(x^{-1}*y)$, and then $U\cap V\not = \emptyset$.

(a) Since $f$ is special, we have that $f(x*y) = f(x)\cdot f(y) = f(y)\cdot f(x)$. Furthermore, $\langle x,y \rangle$ and $\langle f(x),f(y) \rangle$ are commutative by Corollary \ref{cor21}. Hence $f(x)^{-1}\cdot f(x*y) = f(x*y)\cdot f(x)^{-1}$ and by \eqref{eqlema43a}:

\begin{center}
$f(x^{-1}*y*x) = f(x^{-1}*(x*y)) = f(x)^{-1}\cdot f(y) \cdot f(x) = f(x)\cdot (f(x)^{-1}\cdot f(y)) = f(x)\cdot f(y) \cdot f(x)^{-1}$.
\end{center}

(b) Since $f$ is special, we have that $f(y*x^{-1}) = f(y)\cdot f(x)^{-1}$ by Lemma \ref{lema41} and \eqref{sp1}, and then $f(x^{-1}*y) = f(x)^{-1}\cdot f(y)$ by \eqref{sp2}. Putting this in \eqref{eqlema43a} and \eqref{eqlema43b}, we get

\begin{center}
$f(x^{-1}*(x*y)) \in U = \{f(x)^{-1}\cdot (f(x)\cdot f(y)), f(x)\cdot f(y)\cdot f(x)^{-1}\}$,

$f(x*(x^{-1}*y)) \in V = \{f(x)\cdot (f(x)^{-1}\cdot f(y)), f(x)^{-1}\cdot f(y)\cdot f(x)\}$.
\end{center}

Since $f(x)\cdot f(y)\not = f(y)\cdot f(x)$, we have that $f(x)^{-1}\cdot (f(x)\cdot f(y)) \not = f(x)^{-1}\cdot f(y)\cdot f(x)$ and $f(x)\cdot f(y)\cdot f(x)^{-1}\not =  f(x)\cdot (f(x)^{-1}\cdot f(y))$. Furthermore, $f(x)\cdot f(y)\cdot f(x)^{-1} \not = f(x)^{-1}\cdot f(y)\cdot f(x)$ by Corollary \ref{cor31}. Thus 

\begin{center}
 $f(x*(x^{-1}*y)) = f(x)\cdot (f(x)^{-1}\cdot f(y)) = f(x)^{-1}\cdot (f(x)\cdot f(y)) = f(x^{-1}*(x*y))$,
 \end{center} 

and hence $f(y^x) = f(y)^{f(x)}$ and $f(y^{x^{-1}}) = f(y)^{f(x)^{-1}}$ by \eqref{sp2}.

(c) Using (AAIP), we get $f(y^{-1}x^{-1}) = f(x)^{-1} \cdot f(y)^{-1}$. Since $f$ is special, we have $f(y)\cdot f(x)^{-1} \not = f(x)^{-1} \cdot f(y)$ by \eqref{sp1} and Corollary \ref{cor21}, and then $f(x^{-1}*y) =  f(y)\cdot f(x)^{-1}$ by Lemma \ref{lema41}. Putting this in \eqref{eqlema43a} and \eqref{eqlema43b}, we get

\begin{center}
$f(x^{-1}*(x*y)) \in U = \{f(x)^{-1}\cdot f(y)\cdot f(x), (f(y)\cdot f(x))\cdot f(x)^{-1}\},$

$f(x*(x^{-1}*y)) \in V = \{f(x)\cdot f(y)\cdot f(x)^{-1}, (f(y)\cdot f(x)^{-1})\cdot f(x)\}.$
\end{center}

Since $f(x)\cdot f(y)\not = f(y)\cdot f(x)$, we have that $f(x)^{-1}\cdot f(y)\cdot f(x)\not = (f(y)\cdot f(x)^{-1})\cdot f(x)$ and $(f(y)\cdot f(x))\cdot f(x)^{-1} \not = f(x)\cdot f(y)\cdot f(x)^{-1}$. Furthermore, $f(x)^{-1}\cdot f(y)\cdot f(x)\not =  f(x)\cdot f(y)\cdot f(x)^{-1}$ by Corollary \ref{cor31}. Thus 

\begin{center}
$f(x*(x^{-1}*y)) = (f(y)\cdot f(x)^{-1})\cdot f(x) = (f(y)\cdot f(x))\cdot f(x)^{-1} = f(x^{-1}*(x*y))$, 
\end{center}

and hence $f(y^x) = f(y)^{f(x)^{-1}}$ and $f(y^{x^{-1}}) = f(y)^{f(x)}$ by \eqref{sp3}.
\end{proof}

\begin{teo}
\label{teo41} Let $(Q,*)$ and $(Q',\cdot)$ be automorphic loops and $f:(Q,*)\to (Q',\cdot)$ be a half-isomorphism. If $(Q',\cdot)$ satisfies \eqref{co1}, then $f$ is trivial.
\end{teo}
\begin{proof} Suppose by contradiction that $f$ has a GG-triple $(x,y,z)$. Then $x*y\not = y*x$, $x*z\not = z*x$, $f(x*y) = f(x)\cdot f(y)$ and $f(x*z) = f(z)\cdot f(x)$. By Lemma \ref{lema43}, we have that $f(y^x) = f(y)^{f(x)}$, $f(y^{x^{-1}}) = f(y)^{f(x)^{-1}}$, $f(z^x) = f(z)^{f(x)^{-1}}$ and $f(z^{x^{-1}}) = f(z)^{f(x)}$. Using \eqref{phi1}, we get

\begin{center}
$f(y^x)\cdot f(z^{x^{-1}})\cdot f(y^x) = (f(y)\cdot f(z)\cdot f(y))^{f(x)}$ and  $f(y^{x^{-1}})\cdot f(z^{x})\cdot f(y^{x^{-1}}) = (f(y)\cdot f(z)\cdot f(y))^{f(x)^{-1}}$.
\end{center}

Since $f$ is a semi-homomophism, it follows that

\begin{eqnarray}
\label{eqteo41a}
f(y^x*z^{x^{-1}}*y^x) = (f(y*z*y))^{f(x)},\\
\label{eqteo41b}
f(y^{x^{-1}}*z^{x}*y^{x^{-1}}) = (f(y*z*y))^{f(x)^{-1}}.
\end{eqnarray}

If $f(x*(y*z*y)) = f(x)\cdot f(y*z*y)$, then $f(y^x*z^x*y^x) = f((y*z*y)^x) = (f(y*z*y))^{f(x)}$ by \eqref{phi1} and Lemma \ref{lema43}, and so $y^x*z^x*y^x = y^x*z^{x^{-1}}*y^x$ by \eqref{eqteo41a}. Thus $z^x = z^{x^{-1}}$, and hence $x*z=z*x$ by Proposition \ref{prop42a} and Corollary \ref{cor31}, which is a contradiction. When $f(x*(y*z*y)) = f(y*z*y) \cdot f(x)$ we have that $f(y^{x^{-1}}*z^{x^{-1}}*y^{x^{-1}}) = f((y*z*y)^{x^{-1}}) = (f(y*z*y))^{f(x)^{-1}}$ by \eqref{phi1} and Lemma \ref{lema43}, and so $y^{x^{-1}}*z^{x^{-1}}*y^{x^{-1}} = y^{x^{-1}}*z^{x}*y^{x^{-1}}$ by \eqref{eqteo41b}. Thus $z^x = z^{x^{-1}}$, and hence $x*z=z*x$ by Proposition \ref{prop42a} and Corollary \ref{cor31}, which is a contradiction.

Therefore $f$ has no GG-triple and the claim follows by Proposition \ref{prop43}.
\end{proof}

As direct consequences of Theorem \ref{teo41} and Corollaries \ref{cor32} and \ref{cor33}, we have the following results.

\begin{cor}
\label{cor41} Every half-isomorphism between uniquely $2$-divisible automorphic loops is trivial.
\end{cor}

\begin{cor}
\label{cor42} Every half-isomorphism between automorphic loops of odd order is trivial.
\end{cor}

\section{Open problems}

We proved that every half-isomorphism between loops in the class of automorphic loops satisfying \eqref{co1} is special, as was demonstrated earlier for automorphic Moufang loops \cite{GGRS16,KSV}.  
Examples of nontrivial half-isomorphisms between automorphic loops exist (cf. \cite{GA19}), but so far all of them are special. We propose the conjecture that this can be generalized to all automorphic loops:

\begin{conj}
\label{conj1}
Every half-isomorphism between automorphic loops is special.
\end{conj}

\end{document}